\documentclass[12pt,reqno]{amsart}

\usepackage[utf8]{inputenc}
\usepackage[T1]{fontenc}
\usepackage{lmodern}
\usepackage{amsmath,amssymb,amsthm,mathtools}
\usepackage{microtype}
\usepackage[hidelinks]{hyperref}
\hypersetup{
  pdftitle={A sharp Gaussian harmonic-mean inequality for Neumann eigenvalues of the Ornstein--Uhlenbeck operator},
  pdfauthor={Francesco Chiacchio},
  pdfsubject={Sharp reciprocal-sum inequality for Gaussian Neumann eigenvalues},
  pdfkeywords={Ornstein--Uhlenbeck operator, Hermite operator, Neumann eigenvalues, Gaussian measure, harmonic mean, isoperimetric inequality}
}

\allowdisplaybreaks
\numberwithin{equation}{section}

\newtheorem{theorem}{Theorem}[section]
\newtheorem{lemma}[theorem]{Lemma}
\newtheorem{proposition}[theorem]{Proposition}

\theoremstyle{remark}
\newtheorem{remark}[theorem]{Remark}

\newcommand{\R}{\mathbb{R}}
\newcommand{\Sph}{\mathbb{S}}
\newcommand{\dd}{\,\mathrm{d}}
\newcommand{\tr}{\operatorname{tr}}
\newcommand{\Id}{I_N}
\newcommand{\gammaN}{\gamma_N}
\newcommand{\admissible}{\mathcal{A}_N}

\title[Gaussian harmonic-mean inequality]{A sharp Gaussian harmonic-mean inequality for Neumann eigenvalues of the Ornstein--Uhlenbeck operator}

\author{Francesco Chiacchio}
\address{Dipartimento di Matematica e Applicazioni ``R. Caccioppoli'', Universit\`a degli Studi di Napoli Federico II, Complesso Universitario di Monte Sant'Angelo, Via Cintia, 80126 Napoli, Italy}
\email{francesco.chiacchio@unina.it}

\date{}

\begin{document}

\begin{abstract}
Let $N\geq2$ and let $\Omega\subset\R^N$ be a connected Lipschitz domain, possibly unbounded, symmetric with respect to the origin, and such that $0<\gammaN(\Omega)<1$. We assume that the Gaussian Sobolev embedding $H^1(\Omega,\gammaN)\hookrightarrow L^2(\Omega,\gammaN)$ is compact; a sufficient condition is the existence of a bounded Gaussian Sobolev extension operator from $\Omega$ to $\R^N$. Denote by
\[
0=\mu_0(\Omega)<\mu_1(\Omega)\leq\mu_2(\Omega)\leq\cdots
\]
the Neumann eigenvalues of the positive Ornstein--Uhlenbeck operator $-\Delta+x\cdot\nabla$ in $\Omega$. We prove the sharp reciprocal-sum inequality
\[
\sum_{k=1}^{N}\frac{1}{\mu_k(\Omega)}
\geq \frac{N}{\mu_1(B_R)},
\]
where $B_R$ is the Euclidean ball centred at the origin and satisfying $\gammaN(B_R)=\gammaN(\Omega)$. Equality holds if and only if $\Omega=B_R$. The proof combines a coupled $N$-dimensional Ritz argument with a Gaussian raywise rearrangement. The angular imbalance is encoded by a symmetric trace-free matrix, whose contribution is controlled by a finite-dimensional convexity inequality.
\end{abstract}

\maketitle

\section{Introduction}

Let
\begin{equation}\label{eq:gaussian-measure}
\dd\gammaN(x)=(2\pi)^{-N/2}e^{-|x|^2/2}\dd x
\end{equation}
be the standard Gaussian measure on $\R^N$. For an open connected set $\Omega\subset\R^N$, we use the weighted Sobolev space
\[
H^1(\Omega,\gammaN)
=\left\{u\in W^{1,1}_{\mathrm{loc}}(\Omega):
 u,|\nabla u|\in L^2(\Omega,\gammaN)\right\}.
\]

For possibly unbounded domains, Lipschitz regularity and the condition
$0<\gammaN(\Omega)<1$ do not by themselves ensure the compactness needed for
spectral theory. We therefore denote by $\admissible$ the class of connected
Lipschitz domains $\Omega\subset\R^N$ such that
\begin{equation}\label{eq:compact-embedding}
0<\gammaN(\Omega)<1,
\qquad
H^1(\Omega,\gammaN)\hookrightarrow L^2(\Omega,\gammaN)
\quad\text{compactly}.
\end{equation}
A convenient sufficient condition for \eqref{eq:compact-embedding} is the
existence of a bounded linear extension operator
\begin{equation}\label{eq:extension-operator}
\mathcal E_\Omega:H^1(\Omega,\gammaN)
\longrightarrow H^1(\R^N,\gammaN),
\qquad
(\mathcal E_\Omega u)|_\Omega=u,
\end{equation}
with
\[
\|\mathcal E_\Omega u\|_{H^1(\R^N,\gammaN)}
\leq C_\Omega\|u\|_{H^1(\Omega,\gammaN)}.
\]
Indeed, the embedding $H^1(\R^N,\gammaN)\hookrightarrow
L^2(\R^N,\gammaN)$ is compact, and composition with the extension and
restriction operators yields \eqref{eq:compact-embedding}
see 
\cite[Definition~2.1 and Remark~2.1]{ChiacchioGavitone}. In particular, $\admissible$ contains all bounded
Lipschitz domains, all convex domains (possibly unbounded), and all Lipschitz
domains for which a Gaussian Sobolev extension theorem of the form
\eqref{eq:extension-operator} holds; see again
\cite[Remark~2.1]{ChiacchioGavitone}.

On a domain $\Omega\in\admissible$ we consider the Neumann problem
\begin{equation}\label{eq:eigenproblem}
\begin{cases}
-\Delta u+x\cdot\nabla u=\mu u & \text{in }\Omega,\\[2mm]
\dfrac{\partial u}{\partial\nu}=0 & \text{on }\partial\Omega.
\end{cases}
\end{equation}
Its quadratic form is
\[
\mathcal E_\Omega(u,v)=\int_\Omega \nabla u\cdot\nabla v\,\dd\gammaN,
\qquad u,v\in H^1(\Omega,\gammaN).
\]
By \eqref{eq:compact-embedding}, the Neumann problem \eqref{eq:eigenproblem}
has a discrete sequence of eigenvalues. Repeating them according to
multiplicity, we write
\[
0=\mu_0(\Omega)<\mu_1(\Omega)\leq\mu_2(\Omega)\leq\cdots\nearrow+\infty.
\]
The operator in \eqref{eq:eigenproblem} is the positive Ornstein--Uhlenbeck operator. On the whole space its eigenfunctions are the Hermite polynomials, which explains the frequently used terminology \emph{Hermite operator}; see, for instance, \cite{ChiacchioDiBlasio}. Under the Gaussian ground-state transform it is unitarily equivalent to the shifted harmonic oscillator $-\Delta+|x|^2/4-N/2$, although the Neumann condition in \eqref{eq:eigenproblem} becomes a Robin condition for the transformed function.

The classical point of departure is the Szeg\H{o}--Weinberger theorem. If $\Omega\subset\R^N$ is a bounded domain and $B\subset\R^N$ is a ball with $|B|=|\Omega|$, then
\begin{equation}\label{eq:szego-weinberger}
\mu_1(\Omega)\leq\mu_1(B),
\end{equation}
with equality only for balls \cite{Szego,Weinberger}. Szeg\H{o}'s conformal argument in dimension two gives more: for a simply connected planar domain and the disk $D$ of the same area,
\begin{equation}\label{eq:szego-two-term}
\frac{1}{\mu_1(\Omega)}+\frac{1}{\mu_2(\Omega)}
\geq\frac{2}{\mu_1(D)}.
\end{equation}
This observation naturally leads from the optimisation of one eigenvalue to the harmonic mean of the whole first nontrivial eigenspace of the ball.

Ashbaugh and Benguria initiated the systematic study of these reciprocal sums in \cite{AshbaughBenguria}. In the planar case they proved, for arbitrary bounded domains, the universal estimate
\begin{equation}\label{eq:AB-planar}
\frac{1}{\mu_1(\Omega)}+\frac{1}{\mu_2(\Omega)}
\geq\frac{|\Omega|}{2\pi},
\end{equation}
and obtained further sharp conclusions under additional symmetry assumptions. In arbitrary dimensions they established the scale-sharp, but non-optimal, lower bound
\begin{equation}\label{eq:AB-universal}
\sum_{k=1}^{N}\frac{1}{\mu_k(\Omega)}
\geq
\frac{N}{N+2}
\left(\frac{|\Omega|}{\omega_N}\right)^{2/N},
\end{equation}
where $\omega_N$ is the volume of the Euclidean unit ball. They also formulated the sharp conjecture
\begin{equation}\label{eq:AB-conjecture}
\sum_{k=1}^{N}\frac{1}{\mu_k(\Omega)}
\geq\frac{N}{\mu_1(B)},
\qquad |B|=|\Omega|,
\end{equation}
with equality only for balls. The conjecture was subsequently emphasised in Ashbaugh's list of open problems and in Henrot's monograph \cite{AshbaughOpenProblems,Henrot}. For broader accounts of spectral shape optimisation and related extremal problems, we also refer to the survey volume edited by Henrot \cite{HenrotEditedVolume}.

A  partial result was obtained by Xia and Wang, who proved the sharp $(N-1)$-term inequality
\begin{equation}\label{eq:xia-wang-euclidean}
\sum_{k=1}^{N-1}\frac{1}{\mu_k(\Omega)}
\geq\frac{N-1}{\mu_1(B)}
\end{equation}
for Euclidean domains, together with the corresponding hyperbolic result \cite{XiaWang}. The missing last reciprocal term was recovered only recently by He, Li and Tang \cite{HeLiTang}. Their main innovation is to keep the entire $N$-dimensional transplanted eigenspace coupled: instead of estimating the trial functions separately, they estimate the trace of the inverse stiffness matrix times the mass matrix. This leads to a symmetric trace-free defect matrix and a finite-dimensional convexity argument that proves the full conjecture \eqref{eq:AB-conjecture}.

The same sequence of questions has been pursued on spaces of constant curvature. Ashbaugh and Benguria proved the sharp Szeg\H{o}--Weinberger inequality for the first nonzero Neumann eigenvalue on hyperbolic space and, under a hemisphere-containment assumption, on the sphere \cite{AshbaughBenguriaSpaceForms}. For domains contained in a hemisphere of $\Sph^N$, Benguria, Brandolini and the author established the sharp harmonic-mean inequality for the first $N-1$ nontrivial eigenvalues of the Neumann Laplace--Beltrami operator \cite{BenguriaBrandoliniChiacchio}. Xia and Wang obtained the analogous $(N-1)$-term result in hyperbolic space and formulated the full space-form conjecture \cite{XiaWang}. Very recently, You and Zhang proved that conjecture: if $\Omega$ is a smooth bounded domain in a simply connected space form of curvature $-1$ or $1$ (with $\Omega$ contained in an open hemisphere in the spherical case), then
\begin{equation}\label{eq:you-zhang}
\sum_{k=1}^{N}\frac{1}{\mu_k(\Omega)}
\geq\frac{N}{\mu_1(B_\Omega)},
\end{equation}
where $B_\Omega$ is a geodesic ball of the same volume, and equality characterises geodesic balls \cite{YouZhang}. Thus the coupled matrix method now closes the Euclidean, spherical and hyperbolic reciprocal-sum problems.

Complementary results for the first positive Neumann eigenvalue show that the geometry of the spherical problem is subtler outside the hemisphere regime. Langford and Laugesen proved that a geodesic cap maximises $\mu_1$ among simply connected domains occupying up to $94\%$ of $\Sph^2$, thereby extending cap optimality far beyond the hemisphere \cite{LangfordLaugesen}. On the other hand, Bucur, Laugesen, Martinet and Nahon constructed open subsets of $\Sph^2$ of sufficiently large area whose first positive Neumann eigenvalue is strictly larger than that of the equimeasurable geodesic cap \cite{BucurLaugesenMartinetNahon}. Thus cap optimality may persist beyond a hemisphere under a topological restriction, whereas unrestricted spherical Szeg\H{o}--Weinberger statements require genuine geometric or topological hypotheses. These results concern $\mu_1$, rather than the reciprocal sum in \eqref{eq:you-zhang}, but they clarify the role of the assumptions in the spherical theory.

The Gaussian setting is closely related to spherical geometry, but its Neumann optimisation problem displays a strikingly different behaviour. The standard Gaussian measure can be obtained as the Poincar\'e limit of normalised surface measures on high-dimensional spheres. Under the same limiting procedure, spherical caps converge to half-spaces, which are the isoperimetric sets for Gaussian measure. Half-spaces also minimise the first Gaussian Dirichlet eigenvalue. It is therefore tempting to expect the Gaussian isoperimetric set to maximise the first nonzero Neumann eigenvalue, as happens in the classical Euclidean and spherical Szeg\H{o}--Weinberger problems. As already stressed by the author and Di Blasio in \cite{ChiacchioDiBlasio}, this expectation is false.

Indeed, in one dimension, the author and Di Blasio in \cite{ChiacchioDiBlasio}, and Brock, the author and Di Blasio in \cite{BrockChiacchioDiBlasio}, proved by different methods that, among intervals of prescribed Gaussian measure, the first nontrivial Neumann eigenvalue is minimised by a half-line and maximised by the interval centred at the origin; moreover, the eigenvalue varies strictly monotonically as the interval moves between these two configurations. In higher dimensions, the author and Di Blasio showed that, among origin-symmetric domains of prescribed Gaussian measure, the Euclidean ball centred at the origin uniquely maximises $\mu_1$ \cite[Theorem~4.1]{ChiacchioDiBlasio}. The symmetry restriction does not conceal a possible optimality of half-spaces in the unrestricted problem: every Gaussian half-space has first nonzero Neumann eigenvalue equal to $1$, independently of its measure, whereas a suitable smoothing of an explicitly chosen non-symmetric square has first eigenvalue strictly larger than $1$ \cite[Remark~4.3]{ChiacchioDiBlasio}. Thus, even after origin symmetry is removed, half-spaces do not maximise $\mu_1$. On the complementary lower-bound side, Brandolini, the author, Krej\v{c}i\v{r}\'ik and Trombetti proved in the planar case, under containment in a strip, that equality in $\mu_1(\Omega)\geq1$ characterises strips, and also obtained thin-domain convergence results for eigenvalues and eigenfunctions \cite{BrandoliniChiacchioKrejcirikTrombetti}. Beck and Jerison later removed the strip-containment and dimensional restrictions: for every convex domain $K\subset\R^N$, equality holds if and only if, up to rotation, $K=\R\times K'$ for a convex domain $K'\subset\R^{N-1}$; their optimal-transport proof does not contain the thin-domain eigenfunction convergence analysis of the former work \cite[Theorem~I.7]{BeckJerison}.

The lower-order Gaussian problem was advanced by Gao and Wang. In \cite{GaoWang}, they proved that if $\Omega$ is an origin-symmetric Lipschitz domain and $B_R$ is the centred ball satisfying $\gammaN(B_R)=\gammaN(\Omega)$, then
\begin{equation}\label{eq:gao-wang}
\sum_{k=1}^{N-1}\frac{1}{\mu_k(\Omega)}
\geq\frac{N-1}{\mu_1(B_R)},
\end{equation}
with equality only for $B_R$. Their theorem is the Gaussian counterpart of the $(N-1)$-term estimates of Xia and Wang and of Benguria, Brandolini and the author.

The purpose of the present paper is to complete the Gaussian picture by proving the full $N$-term inequality. Although the finite-dimensional mechanism is inspired by He, Li and Tang and is closely related to the subsequent space-form argument of You and Zhang, the Gaussian problem is not a formal consequence of either result. The operator contains a drift term, the measure is not invariant under translations, the radial volume variable has finite total mass, and a translated Weinberger centre would destroy the radial structure of the weight. Origin symmetry is therefore a genuine structural hypothesis: it makes all $N$ transplanted coordinate functions simultaneously orthogonal to the constants. Once this admissibility issue is resolved, the weighted raywise rearrangement produces the same trace-free angular defect that appears in the Euclidean and space-form proofs.

Our main result is the following.

\begin{theorem}\label{thm:main}
Let $N\geq2$, and let $\Omega\in\admissible$ satisfy
\[
\Omega=-\Omega.
\]
Let $B_R$ be the ball centred at the origin such that
\[
\gammaN(B_R)=\gammaN(\Omega).
\]
Then
\begin{equation}\label{eq:main}
\sum_{k=1}^{N}\frac{1}{\mu_k(\Omega)}
\geq \frac{N}{\mu_1(B_R)}.
\end{equation}
Equality holds if and only if $\Omega=B_R$.
\end{theorem}

The proof has three main ingredients. Section~\ref{sec:radial} establishes the radial monotonicity properties of the first nonzero eigenfunctions on a Gaussian ball. Section~\ref{sec:matrix} recalls the trace Ritz principle and the trace-free matrix convexity lemma. Section~\ref{sec:raywise} develops the weighted raywise rearrangement, including its equality case. These ingredients are combined in Section~\ref{sec:proof}. The final appendix studies rectangular boxes by tensorization. Besides giving an exact optimisation result in a non-symmetric class, it makes explicit the spectral role of centring and helps explain why an origin-based hypothesis is natural in the Gaussian problem.

\section{The radial eigenfunction on the comparison ball}\label{sec:radial}

Let $B_R\subset\R^N$ be centred at the origin and set
\[
\lambda=\mu_1(B_R).
\]
By the separation-of-variables analysis in \cite[Lemma~4.1]{ChiacchioDiBlasio}, the first nonzero eigenspace has dimension $N$ and is spanned by
\begin{equation}\label{eq:ball-eigenfunctions}
u_i(r,\theta)=g(r)\theta_i,
\qquad i=1,\ldots,N,
\end{equation}
where $r=|x|$, $\theta=x/|x|\in\Sph^{N-1}$, and $g$ is positive on $(0,R]$ and solves
\begin{equation}\label{eq:radial-ode}
g''(r)+\left(\frac{N-1}{r}-r\right)g'(r)
+\left(\lambda-\frac{N-1}{r^2}\right)g(r)=0
\qquad\text{in }(0,R),
\end{equation}
with the regularity condition at the origin and the Neumann condition
\begin{equation}\label{eq:radial-bc}
g(0)=0,
\qquad g'(R)=0.
\end{equation}
The normalisation of $g$ will be irrelevant.

For later use, we make the behaviour at the singular endpoint explicit. Writing
\[
g(r)=r h\left(\frac{r^2}{2}\right)
\]
reduces \eqref{eq:radial-ode} to Kummer's equation
\[
t h''(t)+\left(\frac N2+1-t\right)h'(t)-\frac{1-\lambda}{2}h(t)=0.
\]
Thus the regular solution can be written as
\begin{equation}\label{eq:kummer-representation}
g(r)=\alpha r\,{}_1F_1\!\left(\frac{1-\lambda}{2};\frac N2+1;\frac{r^2}{2}\right),
\qquad \alpha>0,
\end{equation}
where ${}_1F_1$ is the confluent hypergeometric function; see \cite[Chapter~I]{Tricomi}. In particular,
\begin{equation}\label{eq:origin-expansion}
g(r)=\alpha r+O(r^3),
\qquad
g'(r)=\alpha+O(r^2)
\qquad\text{as }r\downarrow0.
\end{equation}

\begin{lemma}\label{lem:radial-monotonicity}
The radial factor $g$ satisfies
\begin{equation}\label{eq:gprime-positive}
g'(r)>0
\qquad\text{for }0<r<R,
\end{equation}
and
\begin{equation}\label{eq:g-over-r}
\left(\frac{g(r)}{r}\right)'<0
\qquad\text{for }0<r\leq R.
\end{equation}
Moreover,
\begin{equation}\label{eq:lambda-greater-one}
\lambda>1.
\end{equation}
\end{lemma}

\begin{proof}
Set
\[
p(r)=r^{N-1}e^{-r^2/2},
\qquad
W(r)=\frac{N-1}{r^2},
\qquad
\Phi(r)=p(r)g'(r).
\]
Equation \eqref{eq:radial-ode} is equivalent to
\begin{equation}\label{eq:phi-derivative}
\Phi'(r)=\bigl(W(r)-\lambda\bigr)p(r)g(r).
\end{equation}
By \eqref{eq:origin-expansion} and \eqref{eq:radial-bc},
\[
\lim_{r\downarrow0}\Phi(r)=0,
\qquad
\Phi(R)=0.
\]
If $\lambda\leq W(R)$, then $W(r)-\lambda>0$ for every $r\in(0,R)$, because $W$ is strictly decreasing. Since $g>0$, \eqref{eq:phi-derivative} would imply that $\Phi$ is strictly increasing from $0$, contradicting $\Phi(R)=0$. Hence $\lambda>W(R)$.

There is therefore a unique $r_0\in(0,R)$ such that $W(r_0)=\lambda$. From \eqref{eq:phi-derivative},
\[
\Phi'>0\quad\text{in }(0,r_0),
\qquad
\Phi'<0\quad\text{in }(r_0,R).
\]
Together with the endpoint values of $\Phi$, this gives $\Phi(r)>0$ for $0<r<R$, and hence \eqref{eq:gprime-positive}.

Now put
\[
f(r)=\frac{g(r)}{r}.
\]
A direct substitution into \eqref{eq:radial-ode} gives
\begin{equation}\label{eq:f-divergence}
\left(r^{N+1}e^{-r^2/2}f'(r)\right)'
=-(\lambda-1)r^{N+1}e^{-r^2/2}f(r).
\end{equation}
The expansion \eqref{eq:origin-expansion} yields
\[
\lim_{r\downarrow0}r^{N+1}e^{-r^2/2}f'(r)=0.
\]
Assume that $\lambda\leq1$. Then the right-hand side of \eqref{eq:f-divergence} is nonnegative, so $f'\geq0$ in $(0,R)$. Since $f(R)>0$, this would imply
\[
g'(R)=f(R)+Rf'(R)>0,
\]
contrary to \eqref{eq:radial-bc}. Thus \eqref{eq:lambda-greater-one} holds. Integrating \eqref{eq:f-divergence} from $0$ to $r$ now gives
\[
r^{N+1}e^{-r^2/2}f'(r)
=-(\lambda-1)\int_0^r s^{N+1}e^{-s^2/2}f(s)\,\dd s<0,
\]
which proves \eqref{eq:g-over-r}.
\end{proof}

Following the Gaussian Weinberger construction in \cite{ChiacchioDiBlasio,GaoWang}, extend $g$ constantly outside the comparison ball:
\begin{equation}\label{eq:G-definition}
G(r)=
\begin{cases}
g(r),&0\leq r\leq R,\\
g(R),&r\geq R.
\end{cases}
\end{equation}
Since $g'(R)=0$, the function $G$ is of class $C^1$ on $[0,\infty)$. Lemma~\ref{lem:radial-monotonicity} gives
\begin{equation}\label{eq:G-monotonicities}
\begin{aligned}
&G^2 \text{ is nondecreasing on }[0,\infty),\\
&r\longmapsto\frac{G(r)^2}{r^2}
\text{ is strictly decreasing on }(0,\infty).
\end{aligned}
\end{equation}
and $G'(r)=0$ for $r>R$.

Introduce the radial Gaussian measure
\begin{equation}\label{eq:radial-measure}
\dd\nu(r)=(2\pi)^{-N/2}e^{-r^2/2}r^{N-1}\dd r
\end{equation}
and the quantities
\begin{equation}\label{eq:AQH}
\begin{aligned}
A_R&=\int_0^R G(r)^2\,\dd\nu(r),
&Q_R&=\int_0^R G'(r)^2\,\dd\nu(r),\\
H_R&=\int_0^R \frac{G(r)^2}{r^2}\,\dd\nu(r).
\end{aligned}
\end{equation}
Multiplying the divergence form of \eqref{eq:radial-ode} by $g$ and integrating, the boundary terms vanish by \eqref{eq:origin-expansion} and \eqref{eq:radial-bc}. We obtain the energy identity
\begin{equation}\label{eq:energy-identity}
Q_R+(N-1)H_R=\lambda A_R.
\end{equation}

\section{Matrix tools for reciprocal eigenvalue sums}\label{sec:matrix}

We write $A\succeq B$ for real symmetric matrices when $A-B$ is positive semidefinite.

\begin{lemma}\label{lem:trace-ritz}
Let $P_1,\ldots,P_N\in H^1(\Omega,\gammaN)$ be linearly independent in $L^2(\Omega,\gammaN)$ and satisfy
\[
\int_\Omega P_i\,\dd\gammaN=0,
\qquad i=1,\ldots,N.
\]
Define
\begin{equation}\label{eq:mass-stiffness}
M_{ij}=\int_\Omega P_iP_j\,\dd\gammaN,
\qquad
K_{ij}=\int_\Omega \nabla P_i\cdot\nabla P_j\,\dd\gammaN.
\end{equation}
If $K$ is positive definite, then
\begin{equation}\label{eq:ritz-trace}
\sum_{k=1}^{N}\frac{1}{\mu_k(\Omega)}
\geq \tr(K^{-1}M).
\end{equation}
\end{lemma}

\begin{proof}
Let $V=\operatorname{span}\{P_1,\ldots,P_N\}$. The Gram matrix $M$ is positive definite. Let $0<\eta_1\leq\cdots\leq\eta_N$ be the Ritz values obtained by restricting the weighted Neumann Rayleigh quotient to $V$. Since $V$ is contained in the weighted orthogonal complement of the constants, the min--max principle gives
\[
\mu_k(\Omega)\leq\eta_k,
\qquad k=1,\ldots,N.
\]
The $\eta_k$ are the generalized eigenvalues of $Ka=\eta Ma$, and the eigenvalues of $K^{-1/2}MK^{-1/2}$ are their reciprocals. Hence
\[
\sum_{k=1}^{N}\frac{1}{\mu_k(\Omega)}
\geq\sum_{k=1}^{N}\frac{1}{\eta_k}
=\tr(K^{-1/2}MK^{-1/2})
=\tr(K^{-1}M).
\]
This is the trace form of Hersch's variational principle; compare \cite{Hersch,HeLiTang}.
\end{proof}

The following finite-dimensional lemma is the mechanism that recovers the last reciprocal term. It is the trace-free convexity argument introduced in \cite[Lemma~2.2]{HeLiTang} and subsequently used in the space-form setting in \cite{YouZhang}; we include the proof and its equality case.

\begin{lemma}\label{lem:matrix-convexity}
Let $a,c,d,\lambda>0$, and let $Z$ be a real symmetric $N\times N$ matrix with $\tr Z=0$. Suppose that $M$ and $K$ are real symmetric matrices satisfying
\[
M\succeq0,
\qquad
M\succeq a\Id+cZ,
\qquad
0<K\preceq\lambda a\Id-dZ.
\]
Then
\begin{equation}\label{eq:matrix-conclusion}
\tr(K^{-1}M)\geq\frac{N}{\lambda}.
\end{equation}
If equality holds in \eqref{eq:matrix-conclusion}, then
\begin{equation}\label{eq:matrix-equality}
Z=0,
\qquad
M=a\Id,
\qquad
K=\lambda a\Id.
\end{equation}
\end{lemma}

\begin{proof}
Set
\[
T=\lambda a\Id-dZ.
\]
The assumption $0<K\preceq T$ implies $T>0$ and, by the order-reversing property of inversion,
\[
K^{-1}\succeq T^{-1}.
\]
Using $M\succeq0$ and $M-(a\Id+cZ)\succeq0$, we obtain
\begin{equation}\label{eq:matrix-chain}
\tr(K^{-1}M)
\geq\tr(T^{-1}M)
\geq\tr\bigl(T^{-1}(a\Id+cZ)\bigr).
\end{equation}
Diagonalize $Z$ orthogonally and denote its eigenvalues by $z_1,\ldots,z_N$. Since $\sum_i z_i=0$ and $\lambda a-dz_i>0$,
\begin{align*}
\tr\bigl(T^{-1}(a\Id+cZ)\bigr)
&=\sum_{i=1}^{N}\frac{a+cz_i}{\lambda a-dz_i}\\
&=\frac{N}{\lambda}
+\frac{\lambda c+d}{\lambda}
\sum_{i=1}^{N}\frac{z_i}{\lambda a-dz_i}.
\end{align*}
The function
\[
f(z)=\frac{z}{\lambda a-dz}
\]
is strictly convex on $\{z:\lambda a-dz>0\}$, because
\[
f''(z)=\frac{2\lambda ad}{(\lambda a-dz)^3}>0.
\]
Jensen's inequality and $\sum_i z_i=0$ yield
\[
\sum_{i=1}^{N}f(z_i)\geq Nf(0)=0,
\]
which proves \eqref{eq:matrix-conclusion}.

Assume now that equality holds. Strict convexity gives $z_1=\cdots=z_N$, and the trace-free condition yields $Z=0$. Thus $T=\lambda a\Id$. Equality in the second inequality of \eqref{eq:matrix-chain} gives
\[
0=\tr\bigl(T^{-1}(M-a\Id)\bigr).
\]
Since $T^{-1}>0$ and $M-a\Id\succeq0$, it follows that $M=a\Id$. Equality in the first inequality of \eqref{eq:matrix-chain} then gives
\[
0=\tr\bigl((K^{-1}-T^{-1})M\bigr)
=a\tr(K^{-1}-T^{-1}).
\]
The matrix $K^{-1}-T^{-1}$ is positive semidefinite, hence it is zero. Therefore $K=T=\lambda a\Id$.
\end{proof}

\section{Weighted raywise rearrangement}\label{sec:raywise}

We first record a weighted one-dimensional form of the bathtub principle, including the equality case needed below; compare \cite[Theorem~1.14]{LiebLoss} and \cite[Lemma~3.2]{HeLiTang}.

\begin{lemma}\label{lem:bathtub}
Let $0<L\leq\infty$, let $\rho\geq0$ be locally integrable on $[0,L)$, and assume $\rho>0$ almost everywhere in $(0,L)$. Set
\[
\dd\mu(r)=\rho(r)\dd r,
\qquad
V(t)=\mu((0,t)),
\qquad
Y_*=\mu((0,L)).
\]
Then $V$ is continuous and strictly increasing on $[0,L)$. For $0\leq y<Y_*$, let $r_y=V^{-1}(y)$. If $Y_*<\infty$, we also allow $y=Y_*$ and set $r_{Y_*}=L$, with the convention that $(0,L)=(0,\infty)$ when $L=\infty$. If $E\subset(0,L)$ is measurable and $\mu(E)=y$, then:
\begin{enumerate}
\item if $w$ is nondecreasing, then
\[
\int_E w\,\dd\mu\geq\int_0^{r_y}w\,\dd\mu;
\]
\item if $w$ is nonincreasing, then
\[
\int_E w\,\dd\mu\leq\int_0^{r_y}w\,\dd\mu.
\]
\end{enumerate}
If $0<y<Y_*$ and $w$ is strictly decreasing, equality in the second inequality holds only when $E=(0,r_y)$ up to a $\mu$-null set.
\end{lemma}

\begin{proof}
The cases $y=0$ and, when $Y_*<\infty$, $y=Y_*$ are immediate. We therefore assume $0<y<Y_*$. Write $I_y=(0,r_y)$. Since $\mu(E)=\mu(I_y)$,
\[
\mu(E\setminus I_y)=\mu(I_y\setminus E).
\]
If $w$ is nondecreasing, then $w\geq w(r_y)$ on $E\setminus I_y$ and $w\leq w(r_y)$ on $I_y\setminus E$, up to null sets. Consequently,
\begin{align*}
\int_E w\,\dd\mu-\int_{I_y}w\,\dd\mu
&=\int_{E\setminus I_y}w\,\dd\mu
 -\int_{I_y\setminus E}w\,\dd\mu\\
&\geq w(r_y)\bigl(\mu(E\setminus I_y)-\mu(I_y\setminus E)\bigr)=0.
\end{align*}
The nonincreasing case is analogous.

Suppose now that $w$ is strictly decreasing and equality holds in the second inequality. If $\mu(I_y\setminus E)>0$, then also $\mu(E\setminus I_y)>0$. Since $\mu$ is atomless, there is a $\delta>0$ and measurable subsets
\[
A\subset(I_y\setminus E)\cap(0,r_y-\delta),
\qquad
B\subset(E\setminus I_y)\cap(r_y+\delta,L)
\]
of the same positive finite $\mu$-measure. The strict decrease of $w$ gives $w|_A>w|_B$, so replacing $B$ by $A$ strictly increases the integral, a contradiction. Hence $E=I_y$ modulo a null set.
\end{proof}

We now apply the lemma to the radial Gaussian measure \eqref{eq:radial-measure}. Set
\begin{equation}\label{eq:Y-star}
\begin{aligned}
Y_*&=\nu((0,\infty))<\infty,
& V(t)&=\nu((0,t)),\\
\tau&=V^{-1}:[0,Y_*)\longrightarrow[0,\infty).
\end{aligned}
\end{equation}
We extend $\tau$ to the closed interval $[0,Y_*]$ by setting $\tau(Y_*)=\infty$.
For $0\leq y\leq Y_*$, define
\begin{equation}\label{eq:A-H-functions}
A(y)=\int_0^{\tau(y)}G(r)^2\,\dd\nu(r),
\qquad
H(y)=\int_0^{\tau(y)}\frac{G(r)^2}{r^2}\,\dd\nu(r),
\end{equation}
where, at $y=Y_*$, the upper limit is understood as $+\infty$.
For $0<y<Y_*$, differentiation with respect to the volume variable gives
\begin{equation}\label{eq:A-H-derivatives}
A'(y)=G(\tau(y))^2,
\qquad
H'(y)=\frac{G(\tau(y))^2}{\tau(y)^2}.
\end{equation}
By \eqref{eq:G-monotonicities}, $A$ is convex and $H$ is strictly concave on $[0,Y_*)$. Both functions admit finite continuous extensions at $Y_*$, and the corresponding tangent inequalities remain valid at the endpoint. Let
\[
Y_R=V(R)=\nu((0,R)).
\]
For every $y\in[0,Y_*]$, the tangent inequalities at $Y_R$ are
\begin{equation}\label{eq:A-tangent}
A(y)\geq A_R+G(R)^2(y-Y_R)
\end{equation}
and
\begin{equation}\label{eq:H-tangent}
H(y)\leq H_R+\frac{G(R)^2}{R^2}(y-Y_R).
\end{equation}
Since $H'$ is strictly decreasing on $(0,Y_*)$, equality in \eqref{eq:H-tangent} occurs only for $y=Y_R$; this remains true at the endpoint $y=Y_*$ by continuity.

\begin{lemma}\label{lem:raywise-estimates}
Let $E\subset(0,\infty)$ be measurable and set $Y=\nu(E)$. Then
\begin{equation}\label{eq:ray-mass-general}
\int_E G(r)^2\,\dd\nu(r)
\geq A_R+G(R)^2(Y-Y_R),
\end{equation}
\begin{equation}\label{eq:ray-angular-general}
\int_E\frac{G(r)^2}{r^2}\,\dd\nu(r)
\leq H_R+\frac{G(R)^2}{R^2}(Y-Y_R),
\end{equation}
and
\begin{equation}\label{eq:ray-radial-general}
\int_E G'(r)^2\,\dd\nu(r)\leq Q_R.
\end{equation}
If equality holds in \eqref{eq:ray-angular-general}, then equality holds simultaneously in the bathtub bound
\[
\int_E\frac{G(r)^2}{r^2}\,\dd\nu(r)\leq H(Y)
\]
and in the tangent bound \eqref{eq:H-tangent} at $y=Y$.
\end{lemma}

\begin{proof}
By Lemma~\ref{lem:bathtub}, the initial interval $(0,\tau(Y))$ minimises the integral of the nondecreasing function $G^2$ and maximises the integral of the strictly decreasing function $G^2/r^2$. Hence
\[
\int_E G^2\,\dd\nu\geq A(Y),
\qquad
\int_E\frac{G^2}{r^2}\,\dd\nu\leq H(Y).
\]
Combining these inequalities with \eqref{eq:A-tangent} and \eqref{eq:H-tangent} proves \eqref{eq:ray-mass-general} and \eqref{eq:ray-angular-general}. Since $G'=0$ on $(R,\infty)$,
\[
\int_E G'^2\,\dd\nu
\leq\int_0^R G'^2\,\dd\nu=Q_R,
\]
which is \eqref{eq:ray-radial-general}. The last assertion follows because \eqref{eq:ray-angular-general} is obtained from a chain of two inequalities with the same first and last terms.
\end{proof}

\section{Proof of the main theorem}\label{sec:proof}

\begin{proof}[Proof of Theorem~\ref{thm:main}]
For $i=1,\ldots,N$, define
\begin{equation}\label{eq:trial-functions}
P_i(x)=G(|x|)\frac{x_i}{|x|}
\qquad (x\neq0),
\qquad
P_i(0)=0.
\end{equation}
By \eqref{eq:origin-expansion}, $G(r)/r$ and $G'(r)$ are bounded near the origin. For $x\neq0$, writing $r=|x|$ and $\theta=x/r$, one has
\begin{equation}\label{eq:gradient-trial}
|\nabla P_i(x)|^2
=G'(r)^2\theta_i^2
+\frac{G(r)^2}{r^2}(1-\theta_i^2).
\end{equation}
Both $G$ and the functions $G'$ and $G/r$ are bounded on their respective domains. Since $\gammaN(\Omega)<1$, it follows from \eqref{eq:gradient-trial} that $P_i\in H^1(\Omega,\gammaN)$.

Since $\Omega=-\Omega$ and every $P_i$ is odd,
\begin{equation}\label{eq:mean-zero}
\int_\Omega P_i\,\dd\gammaN=0,
\qquad i=1,\ldots,N.
\end{equation}
The functions $P_1,\ldots,P_N$ are linearly independent. Indeed, let $\alpha\in\R^N$. If
\[
P_\alpha=\sum_{i=1}^{N}\alpha_iP_i=0
\quad\text{a.e. in }\Omega,
\]
then continuity gives $P_\alpha=0$ throughout $\Omega$. Away from the origin,
\[
P_\alpha(x)=G(|x|)\frac{\alpha\cdot x}{|x|}.
\]
Since $G(r)>0$ for $r>0$, a nonempty open set could be contained in the zero set of $P_\alpha$ only if $\alpha=0$.

Let $M$ and $K$ be the matrices in \eqref{eq:mass-stiffness}. The matrix $M$ is positive definite. The matrix $K$ is also positive definite: if $\alpha^{\mathsf T}K\alpha=0$, then $P_\alpha$ has zero weak gradient and is constant on the connected domain $\Omega$. Its weighted mean is zero by \eqref{eq:mean-zero}, so $P_\alpha=0$, and linear independence gives $\alpha=0$. Lemma~\ref{lem:trace-ritz} therefore yields
\begin{equation}\label{eq:first-reduction}
\sum_{k=1}^{N}\frac{1}{\mu_k(\Omega)}
\geq\tr(K^{-1}M).
\end{equation}

Let $\dd\sigma$ denote the unnormalised surface measure on $\Sph^{N-1}$. For $\theta\in\Sph^{N-1}$, define the radial section
\begin{equation}\label{eq:radial-section}
E_\theta=\{r>0:r\theta\in\Omega\}
\end{equation}
and its radial Gaussian measure
\begin{equation}\label{eq:Y-theta}
Y(\theta)=\nu(E_\theta).
\end{equation}
For every $\theta$, one has $0\leq Y(\theta)\leq Y_*$; the value $Y(\theta)=Y_*$ is allowed by the endpoint convention introduced in Section~\ref{sec:raywise}. Polar coordinates and $\gammaN(\Omega)=\gammaN(B_R)$ give
\begin{equation}\label{eq:volume-constraint}
\int_{\Sph^{N-1}}\bigl(Y(\theta)-Y_R\bigr)\,\dd\sigma(\theta)=0.
\end{equation}
Define the angular defect matrix
\begin{equation}\label{eq:Z-definition}
Z=\int_{\Sph^{N-1}}
\bigl(Y(\theta)-Y_R\bigr)
\theta\theta^{\mathsf T}\,\dd\sigma(\theta).
\end{equation}
By \eqref{eq:volume-constraint},
\begin{equation}\label{eq:Z-trace}
\tr Z=0.
\end{equation}

For almost every $\theta\in\Sph^{N-1}$, set
\begin{equation}\label{eq:raywise-energies}
\begin{aligned}
\widetilde A_\theta&=\int_{E_\theta}G(r)^2\,\dd\nu(r),
&\widetilde Q_\theta&=\int_{E_\theta}G'(r)^2\,\dd\nu(r),\\
\widetilde H_\theta&=\int_{E_\theta}\frac{G(r)^2}{r^2}\,\dd\nu(r).
\end{aligned}
\end{equation}
In polar coordinates, the mass matrix is
\begin{equation}\label{eq:M-polar}
M=\int_{\Sph^{N-1}}
\widetilde A_\theta\,\theta\theta^{\mathsf T}\,\dd\sigma(\theta).
\end{equation}
Using \eqref{eq:ray-mass-general} and
\begin{equation}\label{eq:spherical-average}
\int_{\Sph^{N-1}}\theta\theta^{\mathsf T}\,\dd\sigma(\theta)
=\frac{|\Sph^{N-1}|}{N}\Id,
\end{equation}
we obtain
\begin{equation}\label{eq:M-bound}
M\succeq a\Id+cZ,
\qquad
a=\frac{|\Sph^{N-1}|}{N}A_R,
\qquad
c=G(R)^2.
\end{equation}

For the stiffness matrix, the vector form of \eqref{eq:gradient-trial} gives
\begin{equation}\label{eq:K-polar}
K=\int_{\Sph^{N-1}}
\left[
\widetilde Q_\theta\,\theta\theta^{\mathsf T}
+\widetilde H_\theta(\Id-\theta\theta^{\mathsf T})
\right]\dd\sigma(\theta).
\end{equation}
Since both $\theta\theta^{\mathsf T}$ and $\Id-\theta\theta^{\mathsf T}$ are positive semidefinite, the scalar estimates \eqref{eq:ray-angular-general}--\eqref{eq:ray-radial-general} imply
\begin{align*}
K
&\preceq Q_R\int_{\Sph^{N-1}}\theta\theta^{\mathsf T}\,\dd\sigma
+H_R\int_{\Sph^{N-1}}(\Id-\theta\theta^{\mathsf T})\,\dd\sigma\\
&\quad+\frac{G(R)^2}{R^2}
\int_{\Sph^{N-1}}
\bigl(Y(\theta)-Y_R\bigr)
(\Id-\theta\theta^{\mathsf T})\,\dd\sigma.
\end{align*}
The scalar part of the last integral vanishes by \eqref{eq:volume-constraint}, and therefore
\[
\int_{\Sph^{N-1}}
\bigl(Y(\theta)-Y_R\bigr)
(\Id-\theta\theta^{\mathsf T})\,\dd\sigma=-Z.
\]
Using \eqref{eq:spherical-average} and the energy identity \eqref{eq:energy-identity}, we conclude that
\begin{equation}\label{eq:K-bound}
K\preceq\lambda a\Id-dZ,
\qquad
d=\frac{G(R)^2}{R^2}>0.
\end{equation}

All the hypotheses of Lemma~\ref{lem:matrix-convexity} are satisfied. Hence
\[
\tr(K^{-1}M)\geq\frac{N}{\lambda}.
\]
Together with \eqref{eq:first-reduction} and $\lambda=\mu_1(B_R)$, this proves \eqref{eq:main}.

We now discuss equality. If $\Omega=B_R$, then
\[
\mu_1(B_R)=\cdots=\mu_N(B_R)=\lambda,
\]
so equality is immediate. Conversely, assume equality in \eqref{eq:main}. Then equality holds throughout
\[
\sum_{k=1}^{N}\frac{1}{\mu_k(\Omega)}
\geq\tr(K^{-1}M)
\geq\frac{N}{\lambda}.
\]
Lemma~\ref{lem:matrix-convexity} gives
\begin{equation}\label{eq:equality-matrices}
Z=0,
\qquad
M=a\Id,
\qquad
K=\lambda a\Id.
\end{equation}
In particular, equality holds in the matrix upper bound \eqref{eq:K-bound}.

For almost every $\theta\in\Sph^{N-1}$, set
\[
L_\theta=H_R+\frac{G(R)^2}{R^2}\bigl(Y(\theta)-Y_R\bigr).
\]
The difference between the right-hand side of \eqref{eq:K-bound} and $K$ can be written as
\begin{equation}\label{eq:K-gap}
\int_{\Sph^{N-1}}
\left[
(Q_R-\widetilde Q_\theta)\theta\theta^{\mathsf T}
+(L_\theta-\widetilde H_\theta)
(\Id-\theta\theta^{\mathsf T})
\right]\dd\sigma(\theta).
\end{equation}
The two scalar coefficients in \eqref{eq:K-gap} are nonnegative by Lemma~\ref{lem:raywise-estimates}. Since the matrix in \eqref{eq:K-gap} is zero, taking its trace gives
\[
0=\int_{\Sph^{N-1}}
\left[(Q_R-\widetilde Q_\theta)
+(N-1)(L_\theta-\widetilde H_\theta)\right]
\dd\sigma(\theta).
\]
The integrand is nonnegative; therefore both terms vanish for almost every $\theta$. In particular,
\begin{equation}\label{eq:angular-equality}
\widetilde H_\theta=L_\theta
\qquad\text{for almost every }\theta.
\end{equation}
For each such direction, the bathtub principle and \eqref{eq:H-tangent} give
\[
\widetilde H_\theta
\leq H(Y(\theta))
\leq L_\theta.
\]
By \eqref{eq:angular-equality}, equality holds in both inequalities. Since $H$ is strictly concave, equality in its tangent inequality at $Y_R$ is possible only when
\begin{equation}\label{eq:Y-equality}
Y(\theta)=Y_R
\qquad\text{for almost every }\theta.
\end{equation}
The equality case in Lemma~\ref{lem:bathtub}, applied to the strictly decreasing function $G^2/r^2$, then yields
\begin{equation}\label{eq:ray-equality}
E_\theta=(0,R)
\quad\text{up to a $\nu$-null set, for almost every }\theta.
\end{equation}
Consequently,
\begin{align*}
\gammaN(\Omega\mathbin\triangle B_R)
&=\int_{\Sph^{N-1}}
\nu\bigl(E_\theta\mathbin\triangle(0,R)\bigr)
\,\dd\sigma(\theta)=0.
\end{align*}
Since the Gaussian density is strictly positive, it also follows that
\begin{equation}\label{eq:lebesgue-null}
|\Omega\mathbin\triangle B_R|=0.
\end{equation}

We finally upgrade \eqref{eq:lebesgue-null} to equality of the open sets. First, $\Omega\subset B_R$. Indeed, if $x\in\Omega\setminus B_R$ and $|x|>R$, a sufficiently small ball around $x$ is contained in $\Omega\setminus B_R$, contradicting \eqref{eq:lebesgue-null}. If $|x|=R$, openness of $\Omega$ gives a ball around $x$ whose intersection with $\R^N\setminus B_R$ has positive measure, leading to the same contradiction.

Conversely, suppose $x\in B_R\setminus\Omega$. If $x\notin\overline\Omega$, then a sufficiently small ball around $x$ is contained in $B_R\setminus\Omega$, again contradicting \eqref{eq:lebesgue-null}. Hence $x\in\partial\Omega$. A Lipschitz domain satisfies an exterior cone condition, and hence an exterior density property, at every boundary point. Thus, for all sufficiently small $r>0$,
\[
|B_r(x)\setminus\Omega|\geq c_x r^N
\]
for some $c_x>0$. Choosing $r<\operatorname{dist}(x,\partial B_R)$ gives a subset of $B_R\setminus\Omega$ of positive measure, contradicting \eqref{eq:lebesgue-null}. Therefore $B_R\subset\Omega$, and $\Omega=B_R$.
\end{proof}

\section{Concluding remarks}

\begin{remark}
The symmetry assumption $\Omega=-\Omega$ is used to guarantee the simultaneous admissibility conditions
\[
\int_\Omega G(|x|)\frac{x_i}{|x|}\,\dd\gammaN=0,
\qquad i=1,\ldots,N.
\]
Clearly, unlike the Euclidean problem, the Gaussian problem is not translation invariant. 
\end{remark}

\begin{remark}
Inequality \eqref{eq:main} is equivalently an upper bound for the harmonic mean of $\mu_1(\Omega),\ldots,\mu_N(\Omega)$. Since
\[
\sum_{k=1}^{N}\frac{1}{\mu_k(\Omega)}
\leq\frac{N}{\mu_1(\Omega)},
\]
Theorem~\ref{thm:main} contains the Gaussian Szeg\H{o}--Weinberger inequality as an immediate consequence.
\end{remark}

\begin{remark}
The unrestricted Gaussian optimisation problem is not expected to select half-spaces. As shown in \cite[Remark~4.3]{ChiacchioDiBlasio}, every half-space has $\mu_1=1$, while non-symmetric bounded domains with $\mu_1>1$ exist. For the standard Gaussian density, however, the optimiser among all non-symmetric domains is still unknown. By contrast, for convex domains the lower-bound equality case is known: the planar strip-contained result of \cite{BrandoliniChiacchioKrejcirikTrombetti} was extended in \cite{BeckJerison} to all dimensions and without any strip-containment assumption, with equality precisely for convex domains that split off a line. 
\end{remark}

\appendix

\section{Rectangular boxes and the role of centring}\label{sec:boxes}

The result in this appendix is not needed in the proof of Theorem~\ref{thm:main}, but it clarifies the meaning of the symmetry assumption. The Gaussian measure has a product structure, and rectangular boxes form a non-symmetric class in which the first nonzero Neumann eigenvalue can be computed by separation of variables. The one-dimensional theorems in \cite{ChiacchioDiBlasio,BrockChiacchioDiBlasio} show that, at fixed Gaussian measure, every displacement of a factor away from the origin is spectrally unfavourable. Tensorization also forces the Gaussian masses of the factors to be balanced. As a consequence, the centred cube is the unique optimal box.

This observation should be interpreted with some care. It does not prove that the centred ball fails to maximise $\mu_1$ among all non-symmetric domains for the standard Gaussian density; that unrestricted problem remains open. It does show, however, that centring at the origin is an active part of the spectral optimisation rather than a harmless normalisation. It corroborates the need for a centring or symmetry hypothesis in general radial-weight Szeg\H{o}--Weinberger type results. The argument below uses the product structure of Gaussian measure rather than the radial matrix comparison employed in the main theorem.

For $s\in(0,1)$, let $a_s>0$ be determined by
\[
\gamma_1((-a_s,a_s))=s,
\]
and set
\[
\mathfrak M(s)=\mu_1((-a_s,a_s)).
\]
By \cite[Theorem~3.1]{ChiacchioDiBlasio}, if $I\subset\R$ is an interval with $\gamma_1(I)=s$, then
\begin{equation}\label{eq:one-dimensional-centering}
\mu_1(I)\leq\mathfrak M(s),
\end{equation}
with equality if and only if $I=(-a_s,a_s)$.

The same centring result, together with the strict monotonicity of the eigenvalue as the interval slides away from the origin while its weighted length is kept fixed, is recovered in a substantially more general form in \cite[Theorem~1.1]{BrockChiacchioDiBlasio}. More precisely, Brock, the author and Di Blasio consider the weighted Neumann problem
\[
-(q u')'=\mu q u\quad\text{in }(a,b),
\qquad u'(a)=u'(b)=0,
\]
where $q\in C^2(\R)$ is positive and even. If $q$ is nonincreasing on $(0,\infty)$, then, among all intervals with prescribed $q$-measure, the centred interval is the unique maximiser of the first nontrivial eigenvalue, provided that $q$ is not constant on the interval. Taking
\[
q(t)=e^{-t^2/2}
\]
recovers exactly \eqref{eq:one-dimensional-centering} and the full Gaussian sliding monotonicity theorem.

The proof in \cite{BrockChiacchioDiBlasio} is more general  than the original Gaussian argument. Thus \eqref{eq:one-dimensional-centering} may equivalently be invoked as the Gaussian special case of \cite[Theorem~1.1]{BrockChiacchioDiBlasio}.

Moreover, $\mathfrak M$ is strictly decreasing as a function of $s$. Indeed, if $\lambda_1^D(J)$ denotes the first Dirichlet eigenvalue of $-\frac{\dd^2}{\dd x^2}+x\frac{\dd}{\dd x}$ on an interval $J$, differentiating a first nonconstant Neumann eigenfunction gives
\[
\mathfrak M(s)=1+\lambda_1^D((-a_s,a_s));
\]
see \cite[(3.9)]{ChiacchioDiBlasio}. Since $s\mapsto a_s$ is strictly increasing, strict Dirichlet domain monotonicity shows that $\mathfrak M$ is strictly decreasing.

\begin{proposition}\label{prop:boxes}
Let
\[
Q=O(I_1\times\cdots\times I_N)\subset\R^N,
\qquad O\in O(N),
\]
be a rectangular box, and let $L=\gammaN(Q)$. Choose $a>0$ so that
\[
\gamma_1((-a,a))=L^{1/N}.
\]
Then
\begin{equation}\label{eq:box-inequality}
\mu_1(Q)\leq\mu_1((-a,a)^N).
\end{equation}
Equality holds if and only if, up to an orthogonal transformation,
\[
Q=(-a,a)^N.
\]
\end{proposition}

\begin{proof}
Rotational invariance allows us to take $O=\Id$. Set
\[
s_j=\gamma_1(I_j),
\qquad j=1,\ldots,N.
\]
The product structure gives
\[
\prod_{j=1}^{N}s_j=L.
\]
The Ornstein--Uhlenbeck operator separates on Cartesian products, so
\[
\mu_1(Q)=\min_{1\leq j\leq N}\mu_1(I_j).
\]
Using \eqref{eq:one-dimensional-centering},
\[
\mu_1(Q)
\leq\min_{1\leq j\leq N}\mathfrak M(s_j)
=\mathfrak M\!\left(\max_{1\leq j\leq N}s_j\right).
\]
Since $\max_j s_j\geq L^{1/N}$ and $\mathfrak M$ is strictly decreasing,
\[
\mu_1(Q)\leq\mathfrak M(L^{1/N})=\mu_1((-a,a)^N),
\]
which proves \eqref{eq:box-inequality}.

If equality holds, strict monotonicity forces $\max_j s_j=L^{1/N}$. Together with $\prod_j s_j=L$, this gives $s_j=L^{1/N}$ for every $j$. Equality in the minimum then requires equality in \eqref{eq:one-dimensional-centering} for each factor, so all intervals are centred. The converse is immediate.
\end{proof}

\end{document}